\documentclass[11pt]{amsart}
\usepackage{latexsym,enumerate}
\usepackage{amsmath,amsthm,latexsym,amssymb,mathrsfs,xcolor}
\usepackage{verbatim}
\usepackage{graphicx}
\usepackage{cite}
\usepackage{comment}

\usepackage{color}

\makeatletter
\def\th@definition{%
  \normalfont 
}
\def\th@plain{%
  \slshape
}
\def\th@remark{%
  \normalfont 
  \thm@preskip\topsep
  \divide\thm@preskip\tw@
  \thm@postskip\thm@preskip
}
\makeatother

\numberwithin{equation}{section}
\theoremstyle{plain}

\newtheorem{theorem}{Theorem}
\newtheorem{lemma}[theorem]{Lemma}
\newtheorem{corollary}[theorem]{Corollary}

\theoremstyle{remark}
\newtheorem{remark}[theorem]{Remark}
\theoremstyle{definition}
\newtheorem{definition}[theorem]{Definition}

\numberwithin{theorem}{section}

\def\R{\mathbb R}
\def\rn{{\mathbb R}^N}

\def\N{\mathbb N}

\def\m2{|\Omega | /2}

\def\be{\begin{equation}}
\def\ee{\end{equation}}

\usepackage{hyperref}
\begin{document}

\title[Weighted Dirichlet-type inequalities]{Weighted Dirichlet-type inequalities for the\\ decreasing rearrangement in cylinders}
\author[F. Brock]{F. Brock$^1$}
\author[F. Chiacchio]{F. Chiacchio$^2$}
\author[A. Ferone]{A. Ferone $^3$}
\author[A. Mercaldo]{A. Mercaldo$^2$}

\normalsize \small
\renewcommand{\baselinestretch}{1.1}
\normalsize
\date{\today}

\setcounter{footnote}{1}
\footnotetext{
Martin-Luther-University of Halle, Landesstudienkolleg, 06114 Halle, Paracelsusstr. 22, Germany,  
 e-mail: {\tt friedemann.brock@studienkolleg.uni-halle.de}
}
\setcounter{footnote}{2}
\footnotetext{Universit\`a di Napoli Federico II, Dipartimento di Matematica e Applicazioni ``R. Caccioppoli'',
Complesso Monte S. Angelo, via Cintia, 80126 Napoli, Italy;\\
e-mail: {\tt fchiacch@unina.it,  mercaldo@unina.it}}

\setcounter{footnote}{3}
\footnotetext
{Universit\`a degli studi della Campania ``L. Vanvitelli'', Dipartimento di Matematica e Fisica, viale Lincoln 5, 81100   Caserta, Italy,
e-mail: {\tt adele.ferone@unicampania.it}
}

\date{March 2026}

 \begin{abstract}
In this paper weighted Dirichlet-type inequalities for the decreasing rearrangement in cylinders are proved. A weighted isoperimetric inequality is also obtained.


\noindent{\sc Mathematics Subject Classification, MSC 2000 : }  26D10, 51M16, 35J20, 35B99
\end{abstract}
 
\maketitle

\section{Introduction}
\setcounter{section}{1}

Rearrangement methods are a classical and powerful tool in the study of variational problems, partial differential equations, and geometric-functional inequalities.  A fundamental example is provided by the so-called Dirichlet-type functional, namely an integral functional of the form
\begin{equation}
\label{generaleq}
J(u) := \int_{\mathbb{R}^N} G(x,u,\nabla u),dx,
\end{equation}
where $u$ is a measurable function with a generalized  first derivatives. Let $u^*$ denote a rearrangement of $u$, such as the Schwarz or Steiner symmetrization. It is well known that, under suitable assumptions, inequalities of the form
\begin{equation}
\label{generaleq2}
J (u^*) \leq J (u),	
\end{equation}
hold. A prototypical example occurs when $J(u)$ is the Sobolev norm of a  function in $W^{1,p}(\mathbb{R}^N)$. In this case, \eqref{generaleq2} reduces to the classical Pólya–Szegő principle, which states that 
Dirichlet integrals  decrease under symmetrization, namely
$$
\int_{\rn} |\nabla u^* |^p\, dx \le \int_{\rn} |\nabla u|^p\,dx,
$$
(see, e.g., \cite{ALT,Br,BanNazarov,Talenti1,Talenti2,ADLT,K1,T}).

Inequalities of this type are closely related to isoperimetric principles and play an important role in the qualitative analysis of variational problems and boundary value problems. In particular, they are often used to derive symmetry properties of minimizers, comparison results, and a priori estimates for solutions.

However, in several situations of current interest,  classical symmetrizations are not the most appropriate tools. This occurs, for example,
when the geometry of the problem favors a fixed direction, or when the functional involves weights or anisotropies. These considerations naturally lead to the search for   inequalities  \eqref{generaleq2} which involve suitable   weighted rearrangements. For example rearrangement inequalities  \eqref{generaleq2} have been extensively investigated  for  Gaussian measure (see, e.g., \cite{Ehrhard, TalentiFerrara}) or more in general  for measures  $d\mu=\varphi (x) dx$ with density of the   form
\begin{equation}
\label{1.4}
\varphi(x) = \psi(x_1)\rho(x_2,\ldots,x_n),
\end{equation}
where $\psi$ and $\rho$ are positive continuous functions (see in \cite{BBMP2}). In that paper a class of weighted “Steiner-type” rearrangements is introduced 
and it is proved that functionals of the form \eqref{generaleq}, defined on sufficiently regular functions, decrease under these weighted rearrangements, provided suitable growth conditions on the integrand are satisfied,
(see also \cite{Ehrhard,TalentiFerrara,BBMP1, Ka1, Br, CiFu, ET}).

Motivated by these developments, in the present paper we investigate weighted Dirichlet-type inequalities associated with the \emph{decreasing rearrangement in one direction}. This type of rearrangement is particularly natural in cylindrical domains and for integral functionals in which one variable plays a distinguished role.

We begin by dealing with the one-dimensional case. If $u^*$ denotes the decreasing rearrangement of a nonnegative function $u$, we prove inequalities of the form
\begin{equation}\label{onedimensional}
\int_I G\bigl(f|u'|,u\bigr)\,dx
\geq
\int_I G\bigl(f|(u^*)'|,u^*\bigr)\,dx,
\end{equation}
for a broad class of continuous integrands $G$ which are convex and nondecreasing in the gradient variable, for any sufficiently regular function $u$ on the intervall  $I=(0,\alpha)$, with $\alpha\in(0,+\infty]$ (see Theorem \ref{Polyaszego1}). We also obtain a more general Landes-type inequality, extending the corresponding result in \cite{landes} (see Theorem \ref{landesgeneral}).

We then turn to the multidimensional case, by considering functions $u$ defined on  cylinders
$
\Omega=\Omega'\times I,
$
where $\Omega'\subset\mathbb{R}^{N-1}$ is a bounded domain with Lipschitz boundary. For these  functions $u=u(x',y)$,  we establish the $N$-dimensional analogue of inequality \eqref{onedimensional}, namely
\begin{equation}
\label{intro:mainineq}
\int_{\Omega} G\bigl(\nabla' u,f|u_y|,u,x'\bigr)\,dx
\geq
\int_{\Omega} G\bigl(\nabla' u^*,f|u_y^*|,u^*,x'\bigr)\,dx,
\end{equation}
where $u^*$ denotes  the decreasing rearrangement with respect to the variable $y$, obtained by rearranging each one-dimensional section $u(x',\cdot)$ with $x'$ fixed. 
We first establish this inequality for a suitable class of piecewise affine functions and then extend it, by approximation, to Lipschitz and Sobolev functions. As a model, we obtain inequalities of the form
\[
\int_{\Omega}\left(|\nabla' u|^2+f(x',y)^2|u_y|^2\right)^{p/2}\,dx
\geq
\int_{\Omega}\left(|\nabla' u^*|^2+f(x',y)^2|u_y^*|^2\right)^{p/2}\,dx.
\]
This yields corresponding inequalities in weighted Sobolev and BV spaces (see Theorems \ref{thnew1} and \ref{BV}),as well as a weighted isoperimetric inequality (see Corollary \ref{cor8}).

Besides their intrinsic interest, inequalities of the form considered here are naturally connected with several concrete analytical and applied problems. A first source of motivation comes from {\sl generalized axially symmetric potential theory}, where non-uniformly elliptic equations in the half-plane arise as reductions of higher-dimensional problems under axial symmetry \cite{WeinsteinGASPT}. Related weighted structures also appear in fluid mechanics, in particular in axially symmetric flow problems, in the study of virtual mass and polarization of bodies of revolution, and in torsion problems for shafts of revolution; see, for instance, \cite{WeinsteinGASPT,WeinsteinFlows,PayneGenElectro,SchifferSzego,WeinsteinTorsion}.

The paper is organized as follows. In Section \ref{sec2} we study the one-dimensional case, proving the weighted isoperimetric inequality for the decreasing rearrangement and deriving the corresponding Dirichlet-type inequalities for functions. In Section \ref{sec3} we establish the multidimensional P\'olya--Szeg\"o principle in cylinders, first for regular functions and then for Sobolev functions. Finally, in Section \ref{sec4} we reformulate the previous results in terms of weighted rearrangements and derive applications to weighted Sobolev and BV spaces.

\section{Weighted  inequalities for the decreasing rearrangement}
\label{sec2}
\setcounter{section}{2}
\setcounter{equation}{0}

Let us begin by introducing some notation. 
\begin{eqnarray*}
&&
\R _+ = (0, +\infty ) , \ \ \overline{\R_+} = [0, +\infty ),
\\
&&
\R ^N \ni x=  (x',y) \ \mbox{ where }\ x' = (x_1 , \ldots , x_{N-1} ) \ \mbox{ and }\ y= x_N ,
\\
&&
\R_+ ^N = \{ x=(x',y ):\, x'\in \R ^{N-1} , \, y>0\} , \quad 
\overline{\R_+ ^N } = \{ x=(x',y ):\, x'\in \R ^{N-1} , \, y\geq 0\},
\\
&&
\mathcal{H}_k \ :  \mbox{  $k$-dimensional Hausdorff measure, $k=0,1,\ldots N$ }.
\end{eqnarray*} 
Throughout this paper 
let $\alpha \in (0, +\infty ] $ and $I$ the interval $(0, \alpha )$.
In this section we assume that 
\begin{eqnarray}
\label{f1}
&&
f\in C(\overline{I}),
\ \ f(x) >0\quad \forall x\in I ,
\\
\label{basiccond} 
&&
f(x_2 -x_1 ) \leq f(x_1 )+f(x_2 ) \ \mbox{ if
$ x_1 , x_2 \in \overline{I}$  with $x_1 \leq  x_2 $, and}
\\
\label{basiccond2} 
&&
\mbox{if $\alpha <+\infty$, then also $
f(x) = f(\alpha -x) \quad \forall x\in \overline{I}$.}
\end{eqnarray}

\begin{lemma}
\label{basiclemma}
Let $x_k \in \overline{I}$, ($k=1, \ldots ,n$), with $x_1 \leq \ldots \leq x_{n} $. Then
\begin{equation}
\label{basicineq}
f\left(
\sum_{k=1} ^{n} x_k (-1)^{n-k} \right) \leq \sum_{k=1} ^{n} f(x_k ) \ .
\end{equation}
\end{lemma}
\begin{proof} The proof of (\ref{basicineq}) is by induction over $n$.
\\
The case $n=1$ is (\ref{basiccond}).
\\
Next assume that (\ref{basicineq}) holds for some $n\in\mathbb{N}$, and let
$x_k \in \overline{I}$, $k= 1, \ldots ,n+1,$ with $ x_1 \leq \ldots \leq x_{n} \leq x_{n+1} $. Then we obtain, using (\ref{basiccond}) and the induction assumption,
$$
f\left( 
\sum_{k=1} ^{n+1} x_k (-1) ^{n+1-k} 
\right)
 \leq  f(x_{n+1} ) + f \left( \sum_{k=1} ^n x_k (-1) ^{n-k} \right)
  \leq  \sum_{k=1} ^{n+1} f(x_k )\ , 
$$
proving (\ref{basicineq}) for $n+1$ in place of $n$.
\end{proof}
If $M\subset \overline{I}$ is Lebesgue measurable with finite measure, then let $|M|$ denote its measure, and 
\begin{equation}
\label{decreasingrearr}
M^*:= (0, |M|) 
\end{equation}
the {\sl decreasing rearrangement of $M$}. 
\\
Following \cite{Camfield}, we now define a perimeter related to the weight function $f$. 
\begin{definition}
Let $M$ be a measurable subset of $I$.
The {\sl $f$-perimeter of $M  $ relative to $I$} is defined by 
\begin{equation}
\label{Pfdef1}
P_f (M, I ) := 
\sup \left\{ \int_{M} \varphi ' (x)\, dx : \ \varphi \in C^1 _0 (I ), \ |\varphi  | \leq f \right\} .
\end{equation}
\end{definition}
It is well-known that the above
definition of the weighted perimeter is equivalent to the following 
\begin{equation}
\label{Pfdef2}
P_{f}(M, I )=
\left\{ 
\begin{array}{ll}
\displaystyle{
\int_{\partial M \cap I  }f (x)\, d\mathcal{H}_0} 
& \mbox{ if } \   \partial M \cap  I  \mbox{ is rectifiable } \\ 
+ \infty   & \mbox{ otherwise }  
\end{array}
\right. 
.
\end{equation}
\begin{theorem}
\label{isopR}
Let $M \subset I $ be measurable with finite measure. 
Then the following isoperimetric inequality holds:
\begin{equation}
\label{isopform}
P_f (M^*, I ) \leq P_f (M, I).
\end{equation}
\end{theorem}
\begin{proof}
It is sufficient to prove (\ref{isopform}) for sets $M\subset I$ which are of the form
$$
M= \bigcup_{j=1}^n (x_{2j-1} , x_{2j}), \ \mbox{ where } \  x_1 < x_2 <\ldots < x_{2n-1} < x_{2n} , \quad (n\in \N ).
$$
Then $M^* = \left( 0, \sum_{k=1} ^{2n} x_k (-1) ^k \right)$.
We split into four cases: 
\\
{\sl (i)} Let $x_1 >0$ and $x_{2n} <\alpha $. Then $P_f (M, I ) = \sum_{k=1} ^{2n} f(x_k ) $ and 
$$
P_f (M^* , I ) = 
 f\left( \sum_{k=1} ^{2n} x_k (-1)^k\right),
$$
so that (\ref{isopform}) follows from (\ref{basicineq}).
\\
{\sl (ii)} Let $x_1 >0$ and $x_{2n} =\alpha <\infty $. Then
$P_f (M, I ) = \sum_{k=1} ^{2n-1} f(x_k ) $. By (\ref{basiccond2}) we obtain
$$
P_f (M^* , I ) = 
 f\left( \alpha -x_{2n-1} +- \ldots +x_2 -x_1 \right) 
=  f\left( x_{2n-1} -+ \ldots -x_2 +x_1 \right) ,
$$
and  (\ref{isopform}) follows from (\ref{basicineq}).
\\
{\sl (iii)} Let $x_1 =0$ and $x_{2n} <\alpha $. Then we have $P_f (M, I ) = \sum_{k=2} ^{2n} f(x_k) $ and 
$$
P_f (M^* , I ) = 
 f\left( \sum_{k=2} ^{2n} x_k (-1)^k \right),
$$
so that (\ref{isopform}) follows again from (\ref{basicineq}).
\\
{\sl (iv)} Let $x_1 =0$ and $x_{2n} =\alpha <+\infty $.
Then we have $P_f (M, I ) = \sum_{k=2} ^{2n-1} f(x_k) $ and 
by (\ref{basiccond2}) we obtain
$$
P_f (M^* , I ) = 
 f\left( \alpha -x_{2n-1} +- \ldots +x_2  \right) 
=  f\left( x_{2n-1} -+ \ldots -x_2 \right) ,
$$
and  (\ref{isopform}) follows again from (\ref{basicineq}).
\end{proof}
\noindent
\begin{remark} 
\label{commentsonf} 
{\bf (a)} The conditions (\ref{basiccond}) and (\ref{basiccond2}) are also necessary for (\ref{isopform}) to hold. To see this, consider first $M= (x_1 , x_2 )$ where $0<x_1 <x_2 <\alpha  $. Then $M^* = (0, x_2 -x_1 )$, so that (\ref{isopform}) implies (\ref{basiccond}).
\\
Next let $M=  (x ,\alpha )$, where
$0\leq x \leq \alpha <+\infty $. Then $M^* = (0, \alpha -x )$, so that (\ref{isopform}) implies 
\begin{equation}
\label{<}
f(\alpha -x ) = P_f (M^* ,I) \leq P_f (M,I)= f(x).
\end{equation}
Further, with the same $x$ and $\alpha $,  let $M= (\alpha -x,\alpha )$. Then $M^* = (0,x)$, so that we obtain from (\ref{isopform})
\begin{equation}
\label{>}
f(x) = P_f (M^*, I) \leq P_f (M,I) = f(\alpha -x).
\end{equation}
Now (\ref{<}) and (\ref{>}) imply (\ref{basiccond2}).
\\ 
{\bf (b)} Let $\alpha =+\infty$. Then (\ref{basiccond}) is for instance satisfied, if $f$ is nondecreasing.
\\
A model case is
\begin{equation}
\label{modelcase}
f(x) = x^{\gamma }, \ (\gamma >0).
\end{equation} 
Moreover, (\ref{basiccond}) also holds for some convex functions which are not monotone, such as
$$
f(x) = x^2 -2x + 2, \quad x\geq 0.
$$
On the other hand, convexity is not sufficient for (\ref{basiccond}), as the example
$$
f(x) = x^2 -2x + \frac{3}{2} , \quad x\geq 0,
$$
shows. 
\\
{\bf (c)} Assume 
(\ref{basiccond2}) and that 
\begin{equation}
\label{basiccond3}
f \ \mbox{ is concave.}
\end{equation}
Then an elementary analysis shows that
\begin{equation}
\label{basiccond4}
f(0) + f(x_2 -x_1 )\leq f(x_1 )+ f(x_2 ) \ \mbox{ whenever $0 \leq x_1 \leq x_2 \leq \alpha $,}
\end{equation} 
so that (\ref{basiccond}) follows.

\noindent Indeed let us define
$a := x_{2} - x_{1} \geq 0,$ $b := x_{1}$ (hence  $a + b = x_{2} \leq \alpha$).
For any fixed $ b$, consider the function
\[
h(t) := f(t + b) - f(t) \quad \text{defined for } 0 \leq t \leq \alpha - b.
\]
Since $f$ is concave, it follows that the function $h$  is non-increasing. Thus, for any $a$ such that $0 \leq a \leq \alpha - b$, we have the following
\[
f(a + b) - f(a) = h(a) \geq h(\alpha - b) = f(\alpha) - f(\alpha - b).
\]
From the symmetry assumption (\ref{basiccond2}), we know that
\[
f(\alpha) = f(0) \quad \text{and} \quad f(\alpha - b) = f(b).
\]
Substituting these into the previous inequality, we obtain
\[
f(a + b) - f(a) \geq f(0) - f(b),
\]
 which becomes
\[
f(x_{2}) - f(x_{2} - x_{1}) \geq f(0) - f(x_{1}).
\]
Finally,  we get the claim
 (\ref{basiccond4}) by addding $f(x_2 - x_1)+ f(x_1)$  to both sides of the last inequality.

\vspace{0.5 cm}

\end{remark}
\hspace*{0.3cm}
By $S $ we denote the class of Lebesgue measurable functions $u: I \to \overline{\R_+}$ satisfying    
$$
|\{ x \in I : \, u(x )>c \}|< +\infty
\quad \mbox{ 
for all $c >0$.}
$$ 
Further, let 
$L $ denote the set of functions $u \in S$ 
which are  Lipschitz  
continuous and have compact support in $\overline{I} $.
\\
Next we give the well-known definition of the 
decreasing rearrangement of functions (see, for example,  \cite{K1}).
\begin{definition}
\label{defdecreasing1} 
Let
$ u \in S $. Then the function $u^* $ defined by
$$
u^* (x ) := \sup \{ 
c > 0 :\, x \in \{ u( \cdot ) > c \}^* \}, \quad (x\in I).
$$
is called the {\sl decreasing rearrangement of $u$}.
\end{definition}
\begin{remark}
\label{rem1onedim} 
By definition $u^*$ is nonincreasing. Moreover we have $\{ u( \cdot ) > c \} ^*  = \{ u^*(\cdot  ) > c\} $.
It is also well-known that if $u$ is continuous, then the  sets $\{ u > c\} $, ($c>0$), are open and the  function $u^* $ has a  pointwise representative which is continuous, too. Moreover, if $u\in L$, then we also have $u^* \in L $. 
\end{remark}
In the following we will use a subclass of functions in $L $, called {\sl nice functions}.
\begin{definition}
\label{nicef1} {\sl (Nice functions)}
A function $u $ is called a {\sl nice function} if $u\in L$, if $u$ is piecewise linear 
 and if the equation $u(x ) = c$ has for every $c > 0$ only a finite  number of solutions $x = x_k $, $k = 1,\ldots , m$, with $m\in \mathbb{N}$.
\end{definition}
\begin{remark}
\label{rem3}  
It is well-known that nice functions are dense in $L $ w.r.t. the norm 
$$
\Vert | u|\Vert := \sup  |u| +\sup |\nabla u|  , 
$$
see e.g. \cite{K1}, pp. 49. A detailed explanation for functions depending on $N$ variables will be given in Remark 3.5. 
\end{remark}
\begin{theorem}
\label{Polyaszego1} Let 
$$
G:   \overline{\R_+} \times  
\overline{\R_+} \to\overline{\R_+} , \ \ G= G( z , v ) ,
$$
be continuous in all the arguments, convex and nondecreasing in $z $. 
Finally, let $u\in L $. Then $u^* \in L $ and there holds
\begin{equation}
\label{dirichletonedim1}
\int_{I  } G\left( f|u' |, u \right) \, dx \geq  
 \int_{I } G \left( f |(u^* ) ^{\prime} |, u^* \right) \, dx .
 \end{equation}  
\end{theorem}
\begin{proof} In view of Remark \ref{rem3} we may restrict ourselves  to nice functions $u$.  We set
$A :=\{ x\in \overline{I} : \, u(x)  >  0 \} $. The set
$$
\{  c >0:\ \exists  x \in A,\ \mbox{  such that } \ u(x)  =  c \}
$$  
can  be  subdivided  into  a  finite number of intervals $U_j $, ($j = 1,\ldots, \kappa $), with the following property: 
\\
For every $ c\in U _j$, the equation $u(x )=c$ has exactly  $n$, ($n = n(j )$, $n\in \mathbb{N}$), solutions $x = x_k ^j $, ($k = 1,..., n$), with
$x^j _1 \leq x^j _2 \leq \ldots \leq x_n ^j$. 
\\
Thus $u$ can be represented in each $U _j$ by the inverse functions 
$$
x = x^j _k ( u), \quad (k = 1,..., n). 
$$
We compute in the domain $U_j $:
\begin{eqnarray*}
u^{\prime} (x^j _k ) & = & \left( \frac{d x_k ^j }{du} \right) ^{-1} \left\{ 
\begin{array}{ll} 
>0 & \mbox{ if $n-k$ is odd}
\\
<0 & \mbox{ if $n-k$ is even} 
\end{array}
\right.
.
\end{eqnarray*} 
Note, that all the derivatives of $x^j_k  $,  ($k  =  1,..., n$), are constant in $U_j $. Therefore $u^* $ is a nice function too, and the equation $u^* = u^*(x)$ has
for  every  $u \in U _j $,  ($j  =  1,..., \kappa $), the (unique)  solution 
$$
x  =  X ^j  = \sum_{k=1} ^{n} x_k ^j (-1) ^{n-k} >0.
$$
This implies 
\begin{eqnarray*}
(u^* )^{\prime} (X^j ) & = & - \left( \sum_{k=1} ^{n} \left| \frac{d x_k ^j }{d u} \right| \right) ^{-1}  \ .
\end{eqnarray*}
The inequality (\ref{dirichletonedim1}) then becomes
\begin{eqnarray}
\label{dirichletonedim2} & &
 \sum_{j=1} ^{\kappa} \int_{U _j } \sum_{k=1} ^{n} G\left(   \frac{f( x_k ^j )}{ \left|d x_k ^j /d u \right|}\right) \cdot \left| \frac{ dx_k ^j }{d u} \right| \, du  
\\
\nonumber
& \geq & 
\sum_{j=1} ^{\kappa} 
\int_{U_j } 
 G\left( 
 \frac{f
 \left( 
 \displaystyle{\sum_{k=1} ^{n} } x_k ^j (-1) ^{n-k} \right)
 }{ 
\displaystyle{\sum_{k=1} ^{n} } \left|d x_k ^j /d u \right|} \right) \cdot \left( \sum_{k=1} ^{n} \left| \frac{d x_k ^j }{ d u}\right| \right)  \, du ,
\end{eqnarray}
 where we omitted the argument $u$ in the functions $G$ and $x_k ^j $. 
\\
To prove (\ref{dirichletonedim2}), it is sufficient to show the corresponding inequality for the integrands for each $u$ lying in one of the intervals $U_j $. 
\\
We  fix  $u$  and $j$  and set  
$x_k := x^j _k $ and
$b_k := |d x_k /d u|$.
Then (\ref{dirichletonedim2}) reduces to
\begin{equation}
\label{dirichletonedim3}
\qquad T := \sum_{k=1} ^n G\left(
\frac{f(x_k)}{b_k} \right) \cdot b_k  \geq
\\
 G\left( 
\frac{
f\left( \sum_{k=1}^n x_k (-1) ^{n-k} 
\right) 
}{
 \sum_{k=1} ^n b_k } \right) \cdot \left( \sum_{k=1} ^n b_k \right) =: T^*. 
\end{equation}    
Since the mapping $z \longmapsto G(z, v )$ is convex and 
nondecreasing,   Jensen's inequality together with Lemma \ref{basiclemma} gives
$$
T \geq  G \left(
\frac{
\sum_{k=1} ^n f( x_k)}{\sum_{k=1} ^n b_k }  
 \right) \cdot \left( \sum_{k=1} ^n b_k \right) 
  \geq T^* .
  $$  
  This proves (\ref{dirichletonedim3}) and therefore (\ref{dirichletonedim2}) .
\end{proof}
\begin{theorem}
\label{landesgeneral}Let 
$$
G:   \overline{\R_+} \times  \overline{\R _+} \to \overline{\R_+} , \ \ G= G( z , v ) ,
$$
be continuous in all the arguments, and assume that 
\begin{eqnarray}
\label{G/z} 
&& G(0,v )=0 \quad \mbox{ and the mapping }\ 
z\longmapsto \frac{G(z,v)}{z} 
\ \mbox{ 
is nondecreasing for every $v\in \R  $. }
\end{eqnarray}
Then for every $u\in L $ we have
\begin{equation}
\label{dirichletlandes} 
\int_{I } G\left( |u^{\prime}|,u \right) \cdot f \, dx \geq 
\int_{I} G\left( |(u^*)^{\prime}|,u \right) \cdot f \, dx.
\end{equation}
\end{theorem}
\begin{proof} We may restrict ourselves to nice functions. Using the notation of the proof of Theorem \ref{Polyaszego1}, inequality (\ref{dirichletlandes}) becomes
\begin{eqnarray}
\label{landes1} & &
 \sum_{j=1} ^{\kappa} \int_{U _j } \sum_{k=1} ^{n} G\left(   \frac{1}{ \left|d x_k ^j /d u \right|}\right) \cdot \left| \frac{ dx_k ^j }{d u} \right| \cdot f(x_k^j )\, du  
\\
\nonumber
& \geq & 
\sum_{j=1} ^{\kappa} 
\int_{U_j } 
 G\left( \frac{1}{ 
\displaystyle{\sum_{k=1} ^{n} } \left|d x_k ^j /d u \right|} \right) \cdot \left( \sum_{k=1} ^{n} \left| \frac{d x_k ^j }{ d u}\right| \right)  \cdot f \left( \sum_{k=1} ^n x_k ^j (-1) ^{n-k} \right)\, du ,
\end{eqnarray}
 where we omitted the argument $u$ in the functions $G$ and $x_k ^j $. 
To prove (\ref{landes1}), it is sufficient to show the corresponding inequality for the integrands, for each $u$ lying in one of the intervals $U_j $. 
\\
Fixing  $u$  and $j$  and setting  
$x_k := x^j _k $ and
$b_k := |d x_k /d u|$,
 (\ref{landes1}) reduces to
\begin{eqnarray}
\label{landes2}
\qquad R & := & \sum_{k=1} ^n G\left(
\frac{1}{b_k} \right) \cdot b_k  \cdot f (x_k )
\\
\nonumber & \geq &
 G\left( 
\frac{1}{
 \sum_{k=1} ^n b_k } \right) \cdot \left( \sum_{k=1} ^n b_k \right) \cdot f\left( \sum_{k=1} ^n x_k (-1) ^{n-k} \right) =: R^*. 
\end{eqnarray}  
By Lemma \ref{basiclemma} and assumption (\ref{G/z}) we obtain  
$$
R\geq G \left( \frac{1}{ \sum_{k=1} ^n b_k } \right) \cdot \left( \sum_{k=1} ^n b_k \right) \cdot \sum_{k=1} ^n f(x_k ) \geq R^* .
$$
This proves (\ref{landes2}) and therefore (\ref{landes1}). 
\end{proof}
\begin{remark}
\label{landes}
Theorem \ref{landesgeneral} generalizes Theorem 2.1, part (ii), of \cite{landes}, where the special case of a nondecreasing $f$ and a convex $G$ was treated.
\end{remark}

\section{Weighted  inequalities on cylinders }
\label{sec3}
\setcounter{section}{3}
\setcounter{equation}{0} 
\noindent
Let us now turn to the $N$-dimensional case. 
Let $\Omega '$ be a bounded domain in $\R ^{N-1 } $ with Lipschitz boundary, and let $\Omega $ be the cylinder  $\Omega ' \times I $. 
\\
By $S ^N$ we denote the class of Lebesgue measurable functions $u: \overline{\Omega } \to \overline{\R _+}$ satisfying
\begin{equation}
\label{require2}   
|\{ y \in I :\, u(x',y )>c \}|< +\infty
\quad \mbox{ 
for all $x'\in \Omega ' $ and all $c > 0 $.}
\end{equation}
Further, let 
$L ^N$ denote the set of functions $u\in S^N$  
which are  Lipschitz  
continuous
and have compact support in $\overline{\Omega}$.
\\
Finally, if $u$ is weakly differentiable, we write for its gradient: 
$$
\nabla u = (\nabla' u, u_y ), \ \mbox{ where }\ \nabla' u =( u_{x_1 } , \ldots , u_{x_{N-1} } ).
$$ 
Next we give the definition of the 
decreasing rearrangement in direction $y$ (see, for example,  \cite{K1}).
Recall, that for any measurable  set $M\subset I $, $M^*$ denotes its decreasing rearrangement.
\begin{definition}
\label{defnonincreasing} 
Let
$ u \in S ^N$. Then the function $u^* $ defined by
$$
u^* (x', y ) := \sup \{ 
c > 0 :\,  y \in \{ u(x', \cdot ) > c \}^* \} , \quad ((x',y) \in \Omega ),
$$
is called the {\sl decreasing rearrangement of $u$ in direction $y$}.
\end{definition}
\begin{remark}
\label{rem1multidim} 
By definition $u^*(x', y )$ is non-increasing in $y$. Moreover we have $\{ u(x', \cdot ) > c \} ^*  = \{ u^*(x',\cdot  ) > c\} $ for every $x' \in \Omega '$ and every $c> 0$.
It is also well-known that if $u$ is continuous, then the  sets $\{ u > c\} $ are open and the  function $u^* $ has a  pointwise representative which is continuous, too. Moreover, if $u\in L ^N$, then we also have $u^* \in L^N$. 
\end{remark}
\begin{remark}
\label{char}
For completeness,
 we define the decreasing rearrangement $M^* $ of a set $M\subset \Omega $ via the rearrangement of its characteristic function:
\begin{equation}
\label{characteristic}
\chi ( M^* ) = \left( \chi (M) \right) ^*.
\end{equation}
\end{remark}
For the proof of the integral inequalities, we need a certain subclass of functions in $L^N $, called {\sl nice functions}.
\begin{definition}
\label{nicefN} {\sl (Nice functions)}
A function $u\in  L ^N $ is called a {\sl nice function} if it is piecewise linear (in the sense of affine), and if the equation $u(x', y ) = c$ has for every $x' \in \Omega ' $ and every 
$c > 0 $ only a finite  number of solutions $y = y_k $, ($k = 1,\ldots , m$, with $m\in \mathbb{N}$).
\end{definition}
\begin{remark}
\label{rem2}  
It is well-known that nice functions are dense in $L ^N$ w.r.t. the norm 
\begin{equation}\label{norm}
 \Vert | u|\Vert := \sup  |u| +\sup |\nabla u|  , 
   \end{equation}
(see e.g. \cite{K1}). For convenience of the reader we present a proof:
Every function in $L^N $ can be approximated by functions in $C ^{\infty } (\overline{\Omega }) \cap L^N $. 
Further, every function in $C^{\infty} (\overline{\Omega})\cap L^N $ can be approximated by nonnegative piecewise linear functions with compact support in $\overline{\Omega }$ . If $u$ is nonnegative and piecewise linear with support in $\overline{\Omega ' } \times [0,R] $, ($R > 0$), then set
\begin{eqnarray*}
\varepsilon _0 & := & \min \left\{ \left| u_y (x) \right| :\ u_y (x) \mbox{ exists and is } \ \not= 0 \right\} >0 ,
\\
v(x) & := & \left\{ 
\begin{array}{ll}
1- y/R  & \mbox{ if }\ 0\leq y <R \\
0 & \mbox{ if }\ y \geq R 
\end{array}
\right. 
\end{eqnarray*}
The functions $u_{\varepsilon} := u + \varepsilon v$ are nice functions if $0 < \varepsilon  < \varepsilon _0$ and $u_{\varepsilon}$ converges to $u$  in the norm \eqref{norm} as $\varepsilon$ tends to zero. 
\end{remark}
\hspace*{0.3cm}In this section we assume
\begin{eqnarray}
\label{newcond}
&& f\in C(\overline{\Omega }), f(x)>0 \ \forall x\in \Omega ,
\\
\label{newcond2} 
&&
f(x', y_2 -y_1 ) \leq f(x', y_1 ) + f(x', y_2 ) \quad \mbox{ if $y_1, y_2 \in \overline{I}$ with $  y_1 \leq  y_2$   and  $ x' \in \overline{\Omega ' } $,}
\\
\label{newcond3}
&&
\mbox{if $\alpha <+\infty $, then also $f(x', y)= f(x', \alpha -x) $ $ \forall (x',y)\in \overline{\Omega }.$}
\end{eqnarray}
\begin{remark}
\label{crucialremark}
Observe that the arguments from Remark \ref{commentsonf} and Lemma \ref{basiclemma} lead to analogous properties for functions $f$ depending on $(x',y)$. In particular, there holds
\begin{eqnarray}
\label{basicineq1}
&& f\left( x', \sum_{k=1} ^n y_k (-1)^{n-k}\right)  \leq  \sum_{k=1} ^n f(x' , y_k )
\\
\nonumber
& & \mbox{whenever }\ x' \in \overline{\Omega ' }\ \mbox{ and } y_1 , \ldots , y_n \in \overline{I} \mbox{ with }  y_1  \leq \ldots \leq y_n .  
\end{eqnarray}
\end{remark}
The following Polya-S\"zego type inequality for Lipschitz continuous functions holds true
\begin{theorem}
\label{PolyaszegoCylinder} Let 
$$
G:  \mathbb{R} ^{N-1} \times \overline{\R_+} \times  \overline{\R _+}  \times \Omega '  \to
\overline{I} , \ \ G= G( z' , z_N , v, x' ) ,
$$
be continuous in all the arguments, convex in $z= (z' , z_N )$ and non-decreasing in $z_N $. 
Finally, let $u\in L^N $. Then $u^* \in L^N $ and there holds
\begin{equation}
\label{dirichlet1}
\int_{\Omega } G\left(\nabla' u , f|u_y |, u , x' \right) \, dx \geq  
 \int_{\Omega } G \left( \nabla' u^* , f |u^* _y |, u^*, x' \right) \, dx .
 \end{equation}  
\end{theorem}
\begin{proof} In view of Remark \ref{rem2} we may restrict ourselves  to nice functions $u$.  We set
$$
A :=\left\{ (x',y) \in \Omega :\, u(x',y)  > 0 \right\} .
$$ 
The domain	
$$
\{ (x', u)\in \mathbb{R} ^{N-1} \times \R _+ : \, \exists  (x', y) \in \Omega ,\ \mbox{  such that } \ u  =  u(x', y ) \}
$$  
can  be  subdivided  into  a  finite number of subdomains $\Omega _j $ , ($j = 1,\ldots, \kappa $), with the following property: 
\\
For every $(x', u)\in\Omega _j$ the equation $u = u(x', y )$ has exactly  $n$, ($n = n(j )$, $n\in \mathbb{N}$), solutions $y = y_k ^j $, ($k = 1,..., n$), with
$0\leq y^j _1 \leq y^j _2 \leq \ldots \leq y_n ^j$. 
\\
Thus $u$ can be represented in each $\Omega _j$ by the inverse functions 
$$
y = y^j _k (x', u), \quad (k = 1,..., n). 
$$
We compute in the domain $\Omega _j$:
\begin{eqnarray*}
u_y (x' , y^j _k ) & = & \left( \frac{\partial y_k ^j }{\partial u} \right) ^{-1} \left\{ 
\begin{array}{ll} 
>0 & \mbox{ if $n-k$ is odd}
\\
<0 & \mbox{ if $n-k$ is even} 
\end{array}
\right.
,
\\
\nabla' u(x' , y_k ^j ) & =& - \nabla' y_k ^j \cdot \left( \frac{\partial y_k ^j}{\partial u} \right) ^{-1} .
\end{eqnarray*} 
Here we omitted the arguments $(x', u)$ in the functions  on  the  right-hand sides. Note, that all the derivatives of $y^j_k  $,  ($k  =  1,..., n$), are constant in $\Omega _j $. Therefore $u^* $ is a nice function too, and the equation $u^* = u^*(x', y )$ has
for  every  $(x', u) \in \Omega _j $,  ($j  =  1,..., \kappa $), the (unique)  solution 
$$
y  =  Y ^j  = \sum_{k=1} ^{n} y_k ^j (-1) ^{n-k} >0.
$$
This implies 
\begin{eqnarray*}
u^* _y (x', Y^j ) & = & - \left( \sum_{k=1} ^{n} \left| \frac{\partial y_k ^j }{\partial u} \right| \right) ^{-1}  \quad \mbox{and}
\\
\nabla' u^* (x', Y^j ) & = & \left( \sum_{k=1} ^{n} \nabla' y_k ^j (-1) ^{n-k+1} \right) \cdot \left(  \sum_{k=1} ^{n} \left| \frac{\partial y_k ^j }{\partial u} \right| \right) ^{-1} .
\end{eqnarray*}
The inequality (\ref{dirichlet1}) then becomes
\begin{eqnarray}
\label{dirichlet2} & &
 \sum_{j=1} ^{\kappa} \int_{\Omega _j } \sum_{k=1} ^{n} G\left(  \frac{\nabla' y_k ^j  (-1)^{n-k+1} }{\left|\partial y_k ^j /\partial u \right|} , \frac{f( y_k ^j )}{ \left|\partial y_k ^j /\partial u \right|}\right) \cdot \left| \frac{\partial y_k ^j }{\partial u} \right|\, dx' \, du  
\\
\nonumber
& \geq & 
\sum_{j=1} ^{\kappa} 
\int_{\Omega_j } 
 G\left( 
 \frac{\displaystyle{\sum_{k=1} ^{n} }
 \nabla' y_k ^j (-1)^{n-k+1}}{ 
 \displaystyle{\sum_{k=1} ^{n} }\left|\partial y_k ^j /\partial u\right| } , 
 \frac{f
 \left( 
 \displaystyle{\sum_{k=1} ^{n} } y_k ^j (-1) ^{n-k} \right)
 }{ 
\displaystyle{\sum_{k=1} ^{n} } \left|\partial y_k ^j /\partial u \right|} \right) \cdot \left( \sum_{k=1} ^{n} \left| \frac{\partial y_k ^j }{ \partial u}\right| \right)  \, dx' \, du ,
\end{eqnarray}
 where we omitted the arguments $x'$ and  $u$ in the functions $G$ , $y_k ^j $ and $f$. 
\\
To prove (\ref{dirichlet2}), it is sufficient to show the corresponding inequality for the integrands for each $(x', u)$ lying in one of the domains $\Omega_j $. 
\\
We  fix $x'$, $u$  and $j$  and set  
$y_k := y^j _k $, $\overrightarrow{c_k }:= \nabla' y_k (-1)^{n-k+1} $ and
$b_k := |\partial y_k /\partial u|$.
Then (\ref{dirichlet2}) reduces to
\begin{equation}
\label{dirichlet3}
\qquad T := \sum_{k=1} ^n G\left( \frac{\overrightarrow{c_k}}{b_k},
\frac{f(y_k)}{b_k} \right) \cdot |b_k|  \geq
\\
 G\left( 
 \frac{\sum_{k=1} ^n \overrightarrow{c_k} 
 }{
 \sum_{k=1} ^n b_k} , 
\frac{
f\left( \sum_{k=1}^n y_k (-1) ^{n-k} 
\right) 
}{
 \sum_{k=1} ^n b_k } \right) \cdot \left( \sum_{k=1} ^n b_k \right) =: T^*. 
\end{equation}   
Since the mapping $(z', z_N) \longmapsto G(z',z_N, v, x' )$ is convex and non-decreasing in the second variable,   Jensen's inequality together with Remark  \ref{crucialremark} gives
$$
T \geq  G \left(  \frac{\sum_{k=1} ^n \overrightarrow{c_k} 
 }{
 \sum_{k=1} ^n b_k} , 
\frac{
\sum_{k=1} ^n f( y_k)}{\sum_{k=1} ^n b_k }  
 \right) \cdot \left( \sum_{k=1} ^n b_k \right) 
  \geq T^* .
  $$  
  This proves (\ref{dirichlet3}) and therefore (\ref{dirichlet1}) .
\end{proof}
\begin{theorem} 
\label{Polyaszego2} Let $G$ be as in Theorem \ref{PolyaszegoCylinder}. Furthermore, assume that $p\in [1, +\infty ) $, $f\in L^\infty(\Omega )$ and 
there exists a positive constant 
 $C>0$ such that 
\begin{equation} 
\label{sobolev}
 0\le G( z' , z_N ,  v,x')  \leq C \left( 1+ |v|^p + |z|^p \right)\ ,
\end{equation}
for all $(z', z_N,  v,x') \in \mathbb{R} ^{N-1} \times 
\overline{\mathbb{R}_+ }  \times \overline{\R _+} \times \Omega ' $.
\\
If $u \in W^{1,p} (\Omega )\cap S^N $, then  $u^*\in  W^{1,p} ( \Omega  )\cap S^N $ and inequality \eqref{dirichlet1} still holds.
\end{theorem}  
\begin{proof} 
{\sl (i)} First we assume that $p\in (1,+\infty )$ and let $u\in   W^{1,p} ( \Omega) \cap S^N $.
Choose a sequence $\{ u_k \} $ of bounded Lipschitz functions such that
$u_k \longrightarrow u $ in $ W ^{1,p } ( \Omega )$.  
Then $u^*_k\in  L^N $ and the following inequalities hold for any $k$:
\begin{equation}
\label{dirichletk}
\int_{\Omega } G\left(\nabla' u_k , f\left |\frac{\partial u_k} {\partial y}\right |, u_k , x' \right) \, dx \geq  
 \int_{\Omega } G \left( \nabla' u_k^* , f \left |\frac{\partial u_k^*} {\partial y}\right |, u_k^*, x' \right) \, dx\,.
 \end{equation} 
 Since $\{u_k\}$ strongly converges to $u$ in $ W ^{1,p } ( \Omega )$,  $u_k$ are equibounded in $ W ^{1,p } ( \Omega )$ and, up to a subsequence, $u_k\rightarrow u$ and $\nabla u_k\rightarrow \nabla u$ a.e. in $\Omega $. Therefore, by using   \eqref{sobolev}, continuity of $G$ and Lebesgue dominated convergence theorem, we get
\begin{equation}\label{lim1}
\lim_{k\to +\infty} \int_{\Omega } G\left(\nabla' u_k , f\left |\frac{\partial u_k} {\partial y}\right |, u_k , x' \right) \, dx= \int_{\Omega } G\left(\nabla' u , f\left |\frac{\partial u} {\partial y}\right |, u, x' \right) \, dx\,.
\end{equation}
On the other hand, by classical properties of rearrangements (see \cite{K1}), we have that
\begin{equation}
u_k^{*}  \rightharpoonup u^{* } \quad \mbox{in }
W^{1,p}(\Omega  )
\label{convsimm}
\end{equation}
and, by lower semicontinuity of the functional, we get
\begin{equation}\label{lim2}
\liminf_k \int_{\Omega } G\left(\nabla' u_k^* , f\left |\frac{\partial u_k^*} {\partial y}\right |, u_k^* , x' \right) \, dx\ge\int_{\Omega } G\left(\nabla' u^*, f\left |\frac{\partial u^*} {\partial y}\right |, u^*, x' \right) \, dx\,.
\end{equation} 
Combining \eqref{dirichletk} - \eqref{lim2}, the conclusion follows.
\\[0.1cm]
{\sl (ii)} 
The proof in the case $p=1$ and $u\in W^{1,1} (\Omega) $ is analogous as in \cite{Br}, Theorem 1.
\end{proof}

\hspace*{0.3cm}In the special case that $G(z', z_N , v, x') = |z|^p $ we obtain the following 
\begin{corollary}
\label{p=2}
Let $1<p<+\infty$, $f\in L^{\infty } (\Omega )$ and $u\in W^{1,p} (\Omega ) \cap S^N $. Then there holds
\begin{equation}
\label{W1p}
\int_{\Omega } \left\{ |\nabla' u|^2 + f^2 (x',y) (u_y) ^2 \right\} ^{p/2} \, dx \geq  
\int_{\Omega  } \left\{ |\nabla' u^*|^2 + f^2 (x',y) (u^*_y) ^2 \right\} ^{p/2} \, dx .
\end{equation}
\end{corollary}

\section{Inequalities involving weighted rearrangement}
\label{sec4}
\setcounter{section}{4}
\setcounter{equation}{0}
\hspace*{0.3cm}If the weight function $f$ is independent of  $x'$ then the results of Section \ref{sec3} can be recast in terms of a {\sl weighted rearrangement}. 
\\
Let $\Omega '$, $I$ and $\Omega $  be as in Section \ref{sec3}, $I_1= (0,\alpha _1 ) $, where $\alpha _1 \in (0, +\infty ]$, and $\Omega _1 = \Omega ' \times I_1 $. Furthermore, let 
\begin{equation}
\label{w1} 
w\in C(\overline{I_1 }), \ \ w(y)>0 \quad \forall y\in I_1 ,  
\end{equation}
and let $W\in C (\overline{I_1 }) $ be its primitive,
\begin{equation}
\label{w2}
W(y):= \int_0 ^y w(t)\, dt , \quad y\in \overline{I_1 }.
\end{equation}
We assume that 
\begin{equation}
\label{w3} 
\alpha = \lim_{y\to \alpha _1 } W(y),
\end{equation}
and that the function $f\in C(\overline{I}) $ defined by 
\begin{equation}
\label{w4}
f(z) := w\left( W^{-1} (z)\right), \quad z\in \overline{I} ,
\end{equation}
satisfies the conditions (\ref{basiccond}) and (\ref{basiccond2}).  
\\
We define function sets $S^N _1 $ and $L^N_1  $ analogously as $S^N$, respectively $L^N $ by replacing $\Omega $ by $\Omega _1  $.
\\
 Now we give the definition of rearrangement with respect to the weight $w$.
\medskip
\begin{definition}
\label{def2} 
Let $v\in S^N _1 $. Setting
\begin{equation}
\label{uv}
u(x', z) := v\left( x', W^{-1} (z)\right) , \quad (x',z)\in  \Omega ,
\end{equation}
let $u^*(x', z)$ be the decreasing rearrangement of $u$ in direction $z$. We call the function $T(v)$, defined by  
\begin{equation}
\label{u*Tv}
T(v)(x', y):= u^* \left( x' , W(y)\right), \quad (x',y)\in \Omega _1 ,
\end{equation}
the {\sl $w$-rearrangement in direction $y$.}
\\
Further, if $M\subset \Omega _1  $ is measurable, we define the {\sl $w$-rearrangement $T(M)$ of $M$} by rearranging the characteristic function of $M$:
\begin{equation}
\label{chiM}
T\left( \chi (M) \right) = \chi (T(M)).
\end{equation} 
\end{definition}
\begin{remark}
\label{rem4} It is easy to see that $T(v) (x',y)$ is non-increasing  in $y$, and the following properties hold:
\begin{equation}
\label{slice}
  \int\limits_{\{a<v(x' , \cdot )<b \} } w(y) \, dy = \int\limits_{\{ a<T(v)(x' , \cdot )<b \} } w(y) \, dy, \quad x' \in \Omega ',
\ 0<a<b.
 \end{equation}
 Further, if $v\in L^N _1 $, then we also have $T(v)\in L^N _1  $.
\end{remark}
Inequality \eqref{dirichlet1} in Theorem \ref{PolyaszegoCylinder} is related to the following inequality for weighted rearrangement via the change of variables
$$
z \longmapsto y= W^{-1} (z).
$$
\begin{theorem}
\label{th2} Let $G$ be as in Theorem \ref{PolyaszegoCylinder} and $v\in L^N _1$. 
Then the following inequality holds 
\begin{equation}
\label{dirichletneu}
\int_{\Omega _1 } G(\nabla' v, |v_y|, v, x')\cdot w(y) \, dx \geq
\int_{\Omega_1 }  G(\nabla' T(v), |(T(v))_y|, T(v), x')\cdot w(y)\, dx.
\end{equation}
where  $v= v(x', y)$ and $T(v) = T(v) (x', y)$.
\end{theorem}
\begin{proof} Apply Theorem \ref{PolyaszegoCylinder} to $u(x', z):= v\left( x', W^{-1} (z)\right)$ and  $f(z):= w(W^{-1} (z))$,  with $(x',z) \in \Omega  $. 
$f$ satisfies the properties (\ref{basiccond}) and (\ref{basiccond2}) and 
we have $u^* (x', z)= T(v) \left( x', W^{-1} (z)\right)$. Hence (\ref{dirichletneu}) follows from (\ref{dirichlet1}).
\end{proof}
\begin{definition}
Let $p\in [1, +\infty )$. 
The space of measurable functions $v: \Omega _1 \to \R $ such that
\begin{equation}
\label{Lp} 
\Vert v\Vert_{p,w} := \left( \int_{\Omega _1 } |v|^p \cdot w \, dx \right) ^{1/p} <+\infty 
\end{equation}
is denoted by $L^p (\Omega _1 ,w)$. 
\\
The weighted Sobolev space $W^{1,p} (\Omega_1 ,w ) $ is the space of functions $v\in L^p (\Omega _1 ,w )$ with weak derivatives $ v_{x_i }$, $i=1, \ldots , N-1$, $v_y $ belonging to $L^p (\Omega _1 ,w)$. A norm in $W^{1,p} (\Omega _1  , w)$ is given by
\begin{equation}
\label{W1pw}
|\Vert v\Vert|_{p,w} := \Vert \nabla v\Vert _{p,w} + \Vert v\Vert_{p,w} .
\end{equation}
\end{definition}
Theorem \ref{th2} can be extended to Sobolev functions. 
\begin{theorem}
\label{thnew1} 
Let $G$ be as in Theorem \ref{Polyaszego2} and let
$v\in W^{1,p} (\Omega _1 ,w)$ be nonnegative. 
\\  
Then $T(v)\in W^{1,p}( \Omega _1 ,w)$ and 
inequality (\ref{dirichletneu})  holds. 
In particular we have
\begin{equation}
\label{normineq}
\Vert \nabla v \Vert _{p,w} \geq \Vert \nabla T(v) \Vert _{p,w} \,.
\end{equation} 
\end{theorem} 
\hspace*{0.3cm}Next we define a weighted BV-space.
\begin{definition}
For any function $v\in L^1 (\Omega _1 ,w  ) $, set
\begin{equation}
\label{BVseminorm}
\Vert \nabla v \Vert _{BV,w} := \sup \left\{ \int_{\Omega _1  } v \cdot \mbox{div}\, \varphi \, dx : \ \varphi \in C_0 ^1 (\Omega _1 , \R ^N), \ |\varphi|\leq w \right\} .
\end{equation} 
The weighted BV-space is the set of all functions $v\in L^1 (\Omega _1 ,w )$ for which the norm
\begin{equation}
\label{BVnorm}
|\Vert v\Vert|_{BV,w} := \Vert v\Vert _{1,w} + 
\Vert \nabla v \Vert _{BV,w}
\end{equation}
is finite.
\end{definition}
\begin{theorem}
\label{BV}
Let $v\in BV (\Omega _1 , w)$ be nonnegative. Then $T(v) \in BV (\Omega _1  , w)$ and 
\begin{equation}
\label{BVineq}
\Vert \nabla v\Vert_{BV,w} \geq \Vert \nabla T(v) \Vert _{BV,w} .
\end{equation}
\end{theorem}
\begin{proof}
Let $v\in BV ( \Omega _1 ,w)$ be nonnegative and let $u(x',z):=v(x', W^{-1}(z))$, ($(x',z) \in \Omega $). 
We can choose a sequence of nonnegative functions $\{ v_k \} \subset W^{1,1}  ( \Omega _1 ,w )$,
such that $v_k \to v $ in $L^1 (\Omega _1 ,w )$ and such that
$\lim_{k\to \infty } \Vert \nabla v_k \Vert _{1,w} = \Vert \nabla v \Vert _{BV,w} $ (see \cite{Camfield}, Theorem 3.2.3 and the introductory remarks to Section 4.1 on p. 80).
Then 
\begin{equation}
\label{L1convergence} 
T(v_k)  \to T(v) \ \mbox{ in }\ L^1 (\Omega _1 ,w  ),
\end{equation} 
and since
\begin{equation}
\label{normineq2}
\Vert \nabla v_k  \Vert _{1,w} \geq \Vert \nabla T(v_k)
\Vert _{1,w} ,\ \ k=1,2, \ldots ,
\end{equation}
by Theorem \ref{thnew1},
the functions $T(v_k ) $ are equibounded in $W^{1,1} (\Omega _1  ,w)$. Since our weight is locally bounded away from zero, Proposition 1.3.1 of \cite{Camfield} tells us that there is a subsequence $\{ k_n \} $ and a vector-valued Radon measure ${\bf m} $, such that 
$$
\nabla T(v_{k_n }) \rightharpoonup {\bf m }  \ \mbox{ weakly in }\   L^1 (\Omega _1  ,w).
$$
In view of (\ref{L1convergence}) this implies that $T(v)\in BV (\Omega _1  , w)$, ${\bf m}= \nabla T(v)$ and moreover,
\begin{equation}
\label{wlscSv}
\Vert \nabla T(v)\Vert_{BV,w} \leq \liminf_{n\to \infty} \Vert \nabla T(v_{k_n }) \Vert _{1,w} .
\end{equation}  
Now the assertion follows from (\ref{normineq2}) and (\ref{wlscSv}). 
\end{proof}
\begin{definition}
Let $M\subset \Omega _1  $ be measurable. Then 
\begin{equation}
\label{wperimeter}
P_w (M, \Omega _1 ) := \Vert \nabla \chi (M)  \Vert _{BV,w} 
\equiv  \sup 
\left\{ \int_{M } \mbox{div}\, {\bf \varphi} \, dx : \ {\bf \varphi} \in C_0 ^1 (\Omega _1 , \R ^N), \ |{\bf \varphi}|\leq w \right\} 
\end{equation}
is called the {\sl $w$-perimeter of $M$}.
\end{definition}
It is well-known that for open bounded sets $M\subset \Omega _1 $ with Lipschitz boundary,  (\ref{wperimeter}) is equivalent to 
\begin{equation}
\label{wperimeter2}
P_w (M, \Omega _1  ) = \int_{\partial M \cap \Omega } w \, d\mathcal{H} _{N-1} .
\end{equation}
From Theorem \ref{BV} applied to the function $\chi(M)$, we have the following isoperimetric inequality for sets of finite $w$-perimeter.
\begin{corollary}\label{cor8}
(Isoperimetric inequality)
Let $M\subset \Omega _1  $ be a set of finite $w$-perimeter. Then
\begin{equation}
P_{w } (M , \Omega _1)  \geq
P_{w} (T(M), \Omega _1  ).
\label{isop2}
\end{equation}
\end{corollary}

\section*{Acknowledgement}

\noindent 
This work was done during the visits made by the first author to
 Dipartimento di Matematica e
Applicazioni ``R. Caccioppoli'' dell' Universit\`a degli Studi di
Napoli Federico II. The author gratefully thanks this institution for its warm hospitality.
\noindent The research of the second author was partially supported by Italian MIUR through research projects PRIN PNRR 2022 - P2022YFAJH - Linear and Nonlinear PDE's: New directions and Applications. The research of the fourth author was partially supported by Italian MIUR through research projects PRIN 2022: PRIN20229M52AS Partial differential equations and related geometric-functional inequalities and PRIN PNRR 2022 - P2022YFAJH - Linear and Nonlinear PDE's: New directions and Applications. The second, third and fourth authors are members of the Gruppo Nazionale per l'Analisi Matematica, la Probabilit\`a
e le loro Applicazioni (GNAMPA) of the Istituto Nazionale di Alta Matematica (INdAM).

\end{document}